\documentclass[11pt,reqno]{amsproc}

\title[]{Global weak solutions for generalized SQG in bounded domains}

\author{Huy Quang Nguyen}
\address{Program in Applied and Computational Mathematics, Princeton University, Princeton, NJ 08544}
\email{qn@math.princeton.edu}

\usepackage[margin=1in]{geometry}
\usepackage{amsmath, amsthm, amssymb, mathrsfs}
\usepackage{times}
\usepackage{color}
\usepackage{hyperref}

\newcommand{\bq}{\begin{equation}}
\newcommand{\eq}{\end{equation}}
\newcommand{\bqa}{\begin{eqnarray*}}
\newcommand{\eqa}{\end{eqnarray*}}
\newcommand{\Rr}{\mathbb{R}}
\newcommand{\Tt}{\mathbb{T}}
\newcommand{\pa}{\partial}
\newcommand{\la}{\label}

\newcommand{\na}{\nabla}
\newcommand{\be}{\begin{equation}}
\newcommand{\ee}{\end{equation}}
\newcommand{\ba}{\begin{array}{l}}
\newcommand{\ea}{\end{array}}

\theoremstyle{plain}
\newtheorem{theo}{Theorem}[section]
\newtheorem{prop}[theo]{Proposition}
\newtheorem{lemm}[theo]{Lemma}

\theoremstyle{definition}
\newtheorem{rema}[theo]{Remark}

\DeclareMathOperator{\cnx}{div}

\DeclareSymbolFont{pletters}{OT1}{cmr}{m}{sl}
\DeclareMathSymbol{s}{\mathalpha}{pletters}{`s}

\def\tt{\theta}

\def\eps{\varepsilon}

\def\na{\nabla}
\def\la{\left\lvert}

\def\le{\leq}

\def\L#1{\langle #1 \rangle}

\def\mez{\frac{1}{2}}

\def\ra{\right\rvert}

\def\P{\mathbb P}

\def\L{\Lambda}

\def\cN{\mathcal{N}}

\numberwithin{equation}{section}

\date{today}
\begin{document}
\begin{abstract}
We prove the existence of global $L^2$ weak solutions for a family of generalized inviscid surface-quasi geostrophic (SQG) equations in bounded domains of the plane. In these equations, the active scalar is transported by a velocity field which is determined by the scalar through a more singular nonlocal operator compared to the SQG equation. The result is obtained by establishing appropriate commutator representations for the weak formulation together with good bounds for them in bounded domains.
\end{abstract}


\noindent

\maketitle
\section{Introduction}
Let $\Omega\subset \Rr^2$ be an open bounded set with smooth boundary. Denote 
\[
\L=(-\Delta)^\mez
\]
where $-\Delta$ is the Laplacian operator in $\Omega$ with homogeneous Dirichlet boundary condition. 

We consider the following family of active scalar equations 
\bq\label{sSQG}
\partial_t\tt  +u\cdot \nabla\tt =0,
\eq
where $\tt = \tt(x,t)$, $u = u(x,t)$ with $(x, t)\in \Omega\times [0, \infty)$ and with the velocity $u$ given by
\bq\label{u:defi}
u = \nabla^\perp \psi,
\eq
\bq\label{psi}
\psi=\L^{-\alpha}\tt,\quad \alpha\in [0, 2].
\eq
Here, fractional powers of the Laplacian $-\Delta$ are based on eigenfunction expansions (see Section \ref{section:frac} below for definitions and notations) and $\psi$ is called the stream function. By \eqref{u:defi} the velocity $u$ is automatically divergence-free. The case $\alpha=2$ corresponds to the 2D Euler equation in the vorticity formulation. When $\alpha=1$, \eqref{sSQG} is the surface-quasi geostrophic (SQG) equation of geophysical significance (\cite{held}), which also serves as a two-dimensional model of the three-dimensional Euler equations in view of many striking physical and mathematical analogies between them (\cite{cmt}). The global regularity issue is known for the 2D Euler equations but remains open for any $\alpha<2$. Growth of solutions when $\alpha=1, 2$ and $\Omega=\Rr^2, \Tt^2$ was studied in \cite{CF}; nonexistence of simple hyperbolic blow-up when $\alpha=1$ and $\Omega=\Rr^2$ was confirmed in \cite{Cor}. We refer to \cite{ccw} for a regularity criterion when $\alpha\in [1, 2]$ and $\Omega=\Rr^2$. On the other hand, it is recently shown in \cite{Kise} finite time blow-up for patch solutions of \eqref{sSQG} in the half plane with small $\alpha<2$. The velocity $u$ becomes more singular when $\alpha$ decreases, and in particular, $u$ is not in $L^2(\Omega)$ if $\tt$ is in $L^2(\Omega)$ and $\alpha<1$. Equations \eqref{sSQG} with $\alpha\in (0, 1)$ were introduced in \cite{cccgw} to understand solutions to the SQG-type equations with even more singular velocity fields. More precisely, it was established in \cite{cccgw} the existence of global $L^2$ weak solutions on the torus $\Tt^2$, together with local existence and uniqueness of strong solutions in $\Rr^2$. The borderline case $\alpha=0$ is surprisingly easy due to the cancellation of the nonlinear term: \eqref{sSQG} reduces to the simple equation $\partial_t\tt=0$, and thus $\tt(\cdot, t)=\tt(\cdot, 0)$ for all $t>0$. On the other hand, if $\alpha<0$ then the stream function $\psi=\L^{-\alpha} \tt$ is not well-defined when $\tt\in L^2(\Omega)$ noticing that there is no dissipation in the equation.

In this paper, we are interested in the issue of global weak solutions for \eqref{sSQG} with $\alpha \in (0, 1)$ in arbitrary (smooth) bounded domains of $\Rr^2$. Let us recall that the existence of global weak solutions for SQG $(\alpha=1$) were first proved in the thesis of Resnick \cite{Res} in the periodic case. This highlights a difference between the nonlinearities of the   SQG equation and the 3D Euler equations: SQG has weak continuity in $L^2$ while the Euler equations do not. The weak continuity of SQG is due to a remarkable commutator structure which was subsequently revisited in \cite{ccw} and used in the proof of absence of anomalous dissipation in \cite{ctv}. In \cite{ConNgu}, this structure was adapted to arbitrary bounded domains to take into account the lack of translation invariance of the fractional Laplacian in domains: a new commutator between the fractional Laplacian and differentiation appears. In addition to that, with the more singular constitutive laws \eqref{psi},  in order to establish the weak continuity of the nonlinearity $u\cdot \nabla\tt$ we will need to find appropriate commutator representations for which good bounds can be derived. Let us emphasize that many known  commutator estimates for fractional Laplacian in the whole space (or on tori) are too expensive for bounded domains due to possible singularity near the boundary or the lack of powerful tools of Fourier analysis. For further results on fractional Laplacian and SQG in bounded domains, we refer to \cite{CabTan, CS, ConIgn, ConIgn2}. 

Our main result is:
\begin{theo}\label{main}
Let $\alpha\in (0, 1)$ and $\tt_0\in L^2(\Omega)$. There exists a weak solution of \eqref{sSQG},  $\tt\in L^\infty([0, \infty); L^2(\Omega))$ with initial data $\tt_0$. That is, for any $T\ge 0$ and $\phi\in C^\infty_0(\Omega\times (0, T))$, $\tt$ satisfies
\bq\label{weak}
\int_0^T\int_\Omega\tt(x, t)\partial_t\phi(x, t)dxdt+\int_0^T\cN(\psi, \phi)dt=0
\eq
and the initial data is attained
\bq\label{weak:data}
\tt(\cdot, 0)=\tt_0(\cdot)\quad\text{in}~H^{-\eps}(\Omega) \quad\forall \eps>0.
\eq
Here,
\bq\label{def:cN}
\cN(\psi, \phi)=\mez \int_\Omega  [\Lambda^\alpha, \nabla^\perp]\psi\cdot \nabla\phi\psi dx-\mez\int_\Omega \L^{-1+\alpha}\nabla^\perp\psi\cdot \L^{1-\alpha}[\Lambda^\alpha,  \nabla\phi] \psi dx.
\eq
Moreover, $\tt$ obeys the energy inequality
\bq\label{energyin}
\Vert \tt(\cdot, t)\Vert^2_{L^2(\Omega)}\le \Vert \tt_0\Vert^2_{L^2(\Omega)}\quad\text{\it a.e.}~t\ge 0.
\eq
Furthermore,  the stream function $\psi\in C([0, \infty); D(\L^{\alpha-\eps}))$ for any $\eps>0$ and its $D(\L^{\frac{\alpha}{2}})$ norm is preserved,
\[
\Vert \psi(\cdot, t)\Vert_{D(\L^{\frac{\alpha}{2}})}=\Vert \psi(\cdot, 0)\Vert_{D(\L^{\frac{\alpha}{2}})}\quad\forall t>0.
\]
\end{theo}
In Theorem \ref{main} and what follows, 
\[
[A, B]:=AB-BA.
\]
denotes the commutator of two operators $A$ and $B$.

When $\alpha=0$, $u=R^\perp \tt$ where $R$ denotes the Riesz transform. As $R: L^2(\Omega)\to L^2(\Omega)$ is continuous, we  have $u\tt\in L^1(\Omega)$ if $\tt\in L^2(\Omega)$. In that case, $\tt$ is a weak solution of \eqref{sSQG} if 
\[
\int_0^T\int_\Omega\tt(x, t)\partial_t\phi(x, t)dxdt+\int_0^T\int_\Omega u(x, t)\tt(x, t)\cdot\nabla \phi(x, t) dxdt=0\quad\forall \phi\in C^\infty_0(\Omega\times (0, T)).
\]
The global existence of such solutions was proved in \cite{ConNgu}. However, when $\alpha<1$, $u$ is less regular then $\tt$ and the second integral in the preceding formulation is not well-defined. Nevertheless, taking into account the nonlinearity structure to explore extra cancellations, this integral has the commutator representation \eqref{def:cN} which makes sense provided only $\tt\in L^2(\Omega)$, as will be proved in Lemma \ref{lemm:rep} below using the heat kernel approach. Let us note that the two objects are equal if $\psi \in H^1_0(\Omega)$, or equivalently, $\tt\in D(\L^{1-\alpha})$. This representation is good enough to well define the nonlinearity but another representation (see \eqref{key3}) will be needed for the compactness argument. The point is that: these two representations are equivalent provided only $\tt\in L^2(\Omega)$ (see Lemma \ref{lemm:key2} below). Unlike the proof in \cite{ConNgu} which uses only Galerkin approximations, Theorem \ref{main} will be proved by a two-tier approximation procedure: Galerkin approximations for each vanishing viscosity approximation. This is because the nonlinearity $u\tt$ is not well-defined in $L^1(\Omega)$ (see Remark \ref{rema:method} below).

The paper is organized as follows. In Section \ref{section:pre}, we present the functional setup of fractional Laplacian in domains and necessary commutator estimates, which can be of independent interest. The proof of Theorem \ref{main} is presented in Section \ref{section:main}. Finally, the proof of the commutator estimates announced in Section \ref{section:pre} are given the appendices.
\section{Preliminaries}\label{section:pre}
\subsection{Fractional Laplacian}\label{section:frac}
Let $\Omega$ be an open bounded set of $\Rr^d$, $d\ge 2$, with smooth boundary. The Laplacian $-\Delta$ is defined on $D(-\Delta)=H^2(\Omega)\cap H^1_0(\Omega)$. Let $\{w_j\}_{j=1}^\infty$ be an orthonormal basis of $L^2(\Omega)$ comprised of $L^2-$normalized eigenfunctions $w_j$ of $-\Delta$, {\it {\it i.e.}}
\[
-\Delta w_j=\lambda_jw_j, \quad \int_{\Omega}w_j^2dx = 1,
\]
with $0<\lambda_1<\lambda_2\le...\le\lambda_j\to \infty$.\\
The fractional Laplacian is defined using eigenfunction expansions,
\[
\Lambda^{s}f\equiv (-\Delta)^{\frac{s}{2}} f:=\sum_{j=1}^\infty\lambda_j^{\frac{s}{2}} f_j w_j\quad\text{with}~f=\sum_{j=1}^\infty f_jw_j,\quad f_j=\int_{\Omega} fw_jdx,
\]
for $s\ge 0$ and 
\[
f\in D(\Lambda^{s}):=\left\{f\in L^2(\Omega): \L^s f\in L^2(\Omega)\right\}.
\]
 The norm of $f=\sum_{j=1}^\infty f_j w_j$ in $D(\Lambda^{s })$, $s\ge 0$, is defined by
\[
\Vert f\Vert_{D(\L^s)}:=\Vert \L^s f\Vert_{L^2(\Omega)}=\big(\sum_{j=1}^\infty\lambda_j^s f_j^2\big)^\mez.
\]
It is also well-known that $D(\Lambda)$ and $H^1_0(\Omega)$ are isometric. 
In the language of interpolation theory, 
\[
D(\Lambda^s)=[L^2(\Omega), D(-\Delta)]_{\frac s2}\quad\forall s\in [0, 2].
\]
As mentioned above,
\[
H^1_0(\Omega)=D(\Lambda)=[L^2(\Omega), D(-\Delta)]_{\mez},
\]
hence
\bq\label{inter}
D(\Lambda^s)=[L^2(\Omega), H^1_0(\Omega)]_{s}\quad\forall s\in [0, 1].
\eq
Consequently, we can identify $D(\Lambda^s)$ with usual Sobolev spaces (see Chapter 1 \cite{LioMag}):
\bq\label{identify}
D(\Lambda^s)=
\begin{cases}
H^s_0(\Omega) &\quad\text{if}~ s\in (\mez, 1],\\
H^\mez_{00}(\Omega):=\{ u\in H^\mez_0(\Omega): u/\sqrt{d(x)}\in L^2(\Omega)\}&\quad\text{if}~ s=\mez,\\
H^s(\Omega) &\quad\text{if}~ s\in [0, \mez).
\end{cases}
\eq
Next, for $s>0$ we define
\[
\L^{-s}f=\sum_{j=1}^\infty \lambda_j^{-\frac{s}{2}}f_jw_j
\]
 if $f=\sum_{j=1}^\infty f_jw_j\in D(\L^{-s})$ with
\[
D(\L^{-s}):=\left\{\sum_{j=1}^\infty f_jw_j\in \mathscr{D}'(\Omega): f_j\in \Rr,~\sum_{j=1}^\infty \lambda_j^{-\frac{s}{2}}f_jw_j\in L^2(\Omega)\right\};
\]
moreover,
\[
 \Vert f\Vert_{D(\L^{-s})}:=\Vert \L^{-s}f\Vert_{L^2(\Omega)}=\big(\sum_{j=1}^\infty\lambda_j^{-s} f_j^2\big)^\mez.
\]
It is easy to check that $D(\L^{-s})$ is the dual of $D(\L^s)$ with respect to the pivot space $L^2(\Omega)$.
 
 We have the following relation between $D(\L^s)$ and $H^s(\Omega)$ when $s\ge 0$.
\begin{prop}\label{prop:inject}
The continuous embedding 
\bq\label{inject}
D(\L^s)\subset H^s(\Omega)
\eq
holds for any $s\ge 0$.
\end{prop}
\begin{proof}
By interpolation, it suffices to prove \eqref{inject} for  $s\in \{0, 1, 2,...\}$. The case $s=0$ is obvious while the case $s=1$ follows from \eqref{identify}. Assume by induction \eqref{inject} for $s\le m$ with $m\ge 1$. Let $\tt\in D(\L^{m+1})$ then $f:=-\Delta \tt\in D(\L^{m-1})$ and thus $f\in H^{m-1}(\Omega)$ by the induction hypothesis. On the other hand, $\tt$ vanishes on the boundary $\partial\Omega$ in the trace sense because $\tt\in D(\L^1)=H^1_0(\Omega)$. Elliptic regularity then implies that $\tt\in H^{m+1}(\Omega)$ and 
\[
\Vert \tt\Vert_{H^{m+1}}\le C\Vert f\Vert_{H^{m-1}}\le C\Vert \Delta \tt\Vert_{m-1, D}=C\Vert \tt\Vert_{m+1, D}
\]
which is \eqref{inject} for $s=m+1$.
\end{proof}
\begin{lemm}
The operator
\bq\label{gradop}
\L^\mu\nabla: D(\L^\gamma)\to D(\L^{\gamma-1-\mu})
\eq
is continuous for any $\gamma\in [0, 1]$ and $\mu\le \gamma-1$.
\end{lemm}
\begin{proof}
We first note that the gradient operator $\nabla$ is continuous from $H^1_0(\Omega)$ to $L^2(\Omega)$ and from $L^2(\Omega)$ to $H^{-1}(\Omega)$, hence by interpolation,
\[
\nabla: [L^2, H^1_0]_\gamma\to [H^{-1}, L^2]_\gamma
\]
for any $\gamma\in [0, 1]$. From the interpolation \eqref{inter} we deduce that
\[
\begin{aligned}
&[L^2, H^1_0]_\gamma=D(\L^\gamma),\\
& [H^{-1}, L^2]_\gamma=\big([H^1, L^2]_\gamma\big)^*=\big([L^2, H^1]_{1-\gamma}\big)^*=D(\L^{1-\gamma})^*=D(\L^{\gamma-1}).
\end{aligned}
\]
Thus, for any $\gamma\in [0, 1]$,
\[
\nabla: D(\L^\gamma)\to D(\L^{\gamma-1})
\]
from which \eqref{gradop} follows.
\end{proof}
\subsection{Commutator estimates}
Here and below $d(x)$ is the distance to the boundary of the domain:
\be
d(x) = d(x, \pa\Omega).
\label{dx}
\ee
Due to the lack of translation invariance,  the fractional Laplacian does not commute with differentiation. The following theorem provides a bound for the commutator.
\begin{theo}[\protect{Theorem 2.2, \cite{ConNgu}}]\label{commu:CN}
Let $p,~q\in [1, \infty]$, $s\in (0, 2)$ and $a$ satisfy 
\[
a(\cdot)d(\cdot)^{-s-1-\frac dp}\in L^q(\Omega).
\]
Then the operator $a[\Lambda^s, \nabla]$ can be uniquely extended from $C^\infty_0(\Omega)$ to $L^p(\Omega)$ such that there exists a positive constant $C=C(d, s, p, \Omega)$ such that
\bq\label{commu:CN:e}
\Vert a[\Lambda^s, \nabla] f\Vert_{L^q(\Omega)}\le C\Vert a(\cdot)d(\cdot)^{-s-1-\frac dp}\Vert_{L^q(\Omega)}\Vert f\Vert_{L^p(\Omega)}
\eq
holds for all $f\in L^p(\Omega)$.
\end{theo}
The bound \eqref{commu:CN:e} is remarkable in that the commutator between an operator of order $s>0$ and an operator of order $1$, which happens to vanish when $\Omega=\Rr^d$, is of order $0$.
\begin{rema}
Let us explain how Theorem \ref{commu:CN} follows from \cite{ConNgu}. Using the heat kernel representation of the fractional Laplacian together with a cancelation of the heat kernel of  $\Rr^d$, it was proved in \cite{ConNgu} the pointwise estimate for $f\in C^\infty_0(\Omega)$
\[
\la [\Lambda^s, \nabla]f(x)\ra\le C(d, s, p, \Omega)d(x)^{-s-1-\frac dp}\Vert f\Vert_{L^p(\Omega)}.
\]
The estimate \eqref{commu:CN:e} then follows by extension by continuity.
\end{rema}
The next commutator estimate for negative powers of Laplacian is needed to handle the situation of more singular velocity.
\begin{theo}\label{commu:new}
Let $s \in (0, d)$ and $a\in W^{1, \infty}(\Omega)$. Let $p, r\in (1, \infty)$ satisfy 
\[
\frac 1p+\frac{d-s}{d}=1+\frac 1r.
\]
Then the operator $[\L^{-s}, a]$ can be uniquely extended from $C^\infty_0(\Omega)$ to $L^p(\Omega)$ with values in $W^{1, r}_0(\Omega)$ such that there exists $C=C(s, d, p, r, \Omega)>0$  such that 
\[
\Vert [\L^{-s}, a] f\Vert_{W^{1, r}_0(\Omega)}\le C\Vert a\Vert_{W^{1, \infty}(\Omega)}\Vert f \Vert_{L^p(\Omega)}
\]
for all $f\in L^p(\Omega)$.\\
In particular, for any $p\in (1, \infty)$, $s\in (0, \frac{d}{p})$, there exists $C=C(s, d, p, \Omega)>0$ such that
\bq\label{cmn:e}
\Vert [\L^{-s}, a] f\Vert_{W^{1, p}_0(\Omega)}\le C\Vert a\Vert_{W^{1, \infty}(\Omega)}\Vert f \Vert_{L^p(\Omega)}
\eq
for all $f\in L^p(\Omega)$.
\end{theo}
With the same method of proof, we obtain
\begin{theo}\label{commu:new2}
Let $s\in (0, 1)$ and $a\in C^\gamma(\Omega)$ with $\gamma\in [0, 1]$ and $s<\gamma$. Let $p, r\in (1, \infty)$ satisfy 
\[
\frac 1p+\frac{d+s-\gamma}{d}=1+\frac 1r.
\]
Then the operator $[\L^s, a]$ can be uniquely extended from $C^\infty_0(\Omega)$ to $L^p(\Omega)$ with values in $L^r(\Omega)$ such that there exists $C=C(s, \gamma, p, r, d, \Omega)>0$  such that 
\bq\label{cmn2:e0}
\Vert [\L^s, a] f\Vert_{L^r(\Omega)}\le C\Vert a\Vert_{C^\gamma(\Omega)}\Vert f \Vert_{L^p(\Omega)}
\eq
for all $f\in L^p(\Omega)$.\\
In particular, for any $p\in (1, \infty)$, if
\[
s\in \big(\max\{\gamma-\frac{d}{p}, 0\}, \max\{\gamma-\frac dp+d, \gamma\}\big)
\]
then there exists $C=C(s, \gamma, p, d, \Omega)>0$ such that
\bq\label{cmn2:e}
\Vert [\L^s, a] f\Vert_{L^p(\Omega)}\le C\Vert a\Vert_{C^\gamma(\Omega)}\Vert f \Vert_{L^p(\Omega)}.
\eq
\end{theo}
\begin{rema}
In view of the identity
\[
\L^{-s}[\L^s, a] f=[a, \L^{-s}]\L^s f,
\]
it follows from \eqref{cmn:e} that 
\bq\label{gain:1-s}
\Vert [\L^s, a]f\Vert_{D(\L^{1-s})}\le C\Vert a\Vert_{W^{1, \infty}(\Omega)}\Vert f\Vert_{D(\L^s)},\quad s\in (0, \frac{d}{2}).
\eq
This exhibits a gain of $1-s$ derivative of $[\L^s, a]$ when acting on $D(\L^s)$. On the other hand, the estimate \eqref{cmn2:e} shows a gain of $s$ derivative when acting on $L^2$.  Both \eqref{cmn:e} and \eqref{cmn2:e} make use of the fact that $\Omega$  is bounded.
\end{rema}
The proofs of Theorems \ref{commu:new}, \ref{commu:new2} are given in the appendices.
\section{Proof of Theorem \ref{main}}\label{section:main}
\subsection{Commutator representations}
First, we adapt the well-known commutator representation of the nonlinearity in SQG (\cite{Res}, see also  \cite{cccgw, CCW, ConNgu}) to take into account the lack of translation invariance of fractional Laplacian and the more singular constitutive law \eqref{psi}:
\begin{lemm} \label{commu:key}
Let $\psi \in H^1_0(\Omega)$,  $u = \na^{\perp}\psi$, and $\theta = \L^{\alpha}\psi$. Let $\phi\in C_0^{\infty}(\Omega)$ be a test function.   Then 
\bq\label{key}
\int_\Omega \tt u\cdot \nabla \phi dx=\mez \int_\Omega  [\Lambda^\alpha, \nabla^\perp]\psi\cdot \nabla\phi\psi dx-\mez\int_\Omega  \nabla^\perp\psi\cdot [\Lambda^{\alpha}, \nabla\phi] \psi dx
\eq
holds.
\end{lemm}
\begin{proof}
We have
\[
\int_\Omega \tt u\cdot \nabla \phi dx =\int_{\Omega}\Lambda^{\alpha} \psi\nabla^\perp\psi\cdot \nabla \phi dx  = -\int_{\Omega}\psi\nabla^\perp\L^{\alpha}\psi\cdot \nabla \phi dx,
\]
where we integrated by parts and used the fact that $\na^{\perp}\cdot\na\phi = 0$. The first and middle terms are well defined because $\theta u=\theta \nabla^\perp \psi\in L^1(\Omega)$ noticing that $\psi\in H^1_0(\Omega)$ and $\theta=\L^{\alpha}\psi\in D(\L^{1-\alpha})\subset L^2(\Omega)$. The last term is defined because $\na\phi\cdot \na^{\perp}\L^{\alpha}\psi\in H^{-1}(\Omega)$ and $\psi\in H_0^1(\Omega)$. Commuting $\nabla^\perp$ with $\Lambda^{\alpha}$ and then with $\nabla \phi$ leads to
\[
\begin{aligned}
\int_\Omega \tt u\cdot \nabla \phi dx&=-\int_\Omega   \psi [\nabla^\perp,\Lambda^{\alpha}]\psi \cdot \nabla\phi  dx-\int_\Omega   \psi \Lambda^{\alpha} \nabla^\perp\psi\cdot \nabla\phi dx\\
&=-\int_\Omega  \psi [\nabla^\perp,\Lambda^{\alpha}]\psi\cdot \nabla\phi dx-\int_\Omega  \nabla^\perp\psi\cdot \Lambda^{\alpha}(\psi \nabla\phi) dx\\
&=-\int_\Omega  [\nabla^\perp,\Lambda^{\alpha}]\psi\cdot \nabla\phi\psi dx-\int_\Omega  \nabla^\perp\psi\cdot [\Lambda^{\alpha}, \nabla\phi] \psi dx-\int_\Omega  \nabla^\perp\psi\cdot \nabla \phi \Lambda^{\alpha}\psi dx\\
&= -\int_\Omega  [\nabla^\perp,\Lambda^{\alpha}]\psi\cdot \nabla\phi\psi dx-\int_\Omega  \nabla^\perp\psi\cdot [\Lambda^{\alpha}, \nabla\phi] \psi dx-\int_\Omega  \tt u\cdot \nabla \phi dx.
\end{aligned}
\]
The above calculations are justified by means of Theorems \ref{commu:CN} and \ref{commu:new2}. Noticing that the last term on the right-hand side is exactly the negative of the left-hand side, we proved \eqref{key}.
\end{proof}
\begin{rema}
The representation \eqref{key} was derived in \cite{ConNgu} for the SQG equation ($\alpha=1$). When $\Omega=\Rr^2$ or $\Tt^2$, \eqref{key} reduces to 
\[
\int_\Omega \tt u\cdot \nabla \phi dx=-\mez\int_\Omega  \nabla^\perp\psi\cdot [\Lambda^{\alpha}, \nabla\phi] \psi dx.
\]
Integrating by parts yields
\[
-\mez\int_\Omega  \nabla^\perp\psi\cdot [\Lambda^{\alpha}, \nabla\phi] \psi dx=\mez\int_\Omega \psi\cdot \nabla^\perp[\Lambda^{\alpha}, \nabla\phi] \psi dx=\mez\int_\Omega \psi\cdot [\Lambda^{\alpha}\nabla^\perp, \nabla\phi] \psi dx
\]
where we used in the second equality the fact that $\nabla^\perp\cdot\nabla\phi=0$. This representation was invoked in \cite{cccgw} to prove the  existence of global $L^2$ weak solutions of \eqref{sSQG} in the periodic setting. More precisely, the authors proved the commutator estimate
\[
\Vert [\L^s\nabla, g]h\Vert_{L^2(\Tt^2)}\le C\Vert h\Vert_{L^2(\Tt^2)}\Vert g\Vert_{H^{s+2+\eps}(\Tt^2)}+C\Vert \L^sh\Vert_{L^2(\Tt^2)}\Vert g\Vert_{H^{2+\eps}(\Tt^2)}
\]
for any $s, \eps>0$. In arbitrary bounded domains, we were not able to establish such a commutator estimate.
\end{rema}
We observe that by virtue of Theorem \ref{commu:CN}, the first integral in \eqref{key} is well-defined provided only $\psi \in L^2(\Omega)$; moreover,
\[
\la \int_\Omega  [\Lambda^{-\alpha}, \nabla^\perp]\psi\cdot \nabla\phi\psi dx\ra \le C\Vert \nabla\phi d(\cdot)^{-\alpha-2}\Vert_{L^2(\Omega)}\Vert \psi\Vert_{L^2(\Omega)}^2
\]
where by applying three times the Hardy inequality we get
\[
\Vert \nabla\phi d(\cdot)^{-\alpha-2}\Vert_{L^2(\Omega)}\le C\Vert \nabla\phi d(\cdot)^{-3}\Vert_{L^2}\le C\Vert \nabla^4\phi\Vert_{L^2(\Omega)}\le C\Vert \phi\Vert_{H^4(\Omega)}.
\]
Consequently,
\bq\label{cN1:bound}
\la \int_\Omega  [\Lambda^{\alpha}, \nabla^\perp]\psi\cdot \nabla\phi\psi dx\ra\le   C\Vert \phi\Vert_{H^4(\Omega)}\Vert \psi\Vert_{L^2(\Omega)}^2.
\eq
Regarding the second integral, we prove
\begin{lemm}\label{lemm:key2}
Assume $\psi\in D(\L^\alpha)$. Then
\bq\label{key2}
\cN_2(\psi, \phi):=\int_\Omega \L^{-1+\alpha}\nabla^\perp\psi\cdot \L^{1-\alpha}[\Lambda^\alpha, \nabla\phi] \psi dx
\eq
satisfies
\bq\label{cN2:bound}
\la \cN_2(\psi, \phi)\ra \le C\Vert \nabla\phi\Vert_{W^{1, \infty}}\Vert \psi\Vert^2_{D(\L^\alpha)}.
\eq
For any $\delta\in (0, \min(\alpha, 1-\alpha))$ we have
\bq\label{key3}
\begin{aligned}
\cN_2(\psi, \phi)&=\int_\Omega  \L^{-1+\alpha-\delta}\nabla^\perp\psi\cdot \L[\nabla\phi, \L^{-\alpha+\delta}]\L^\alpha \psi dx
+\int_\Omega  \L^{-1+\alpha}\nabla^\perp\psi\cdot \L[\nabla\phi , \L^{-\delta}]\L^\delta\psi dx.
\end{aligned}
\eq
Moreover,
\bq\label{cN2d:bound}
\la \cN_2(\psi, \phi)\ra \le C\Vert \nabla\phi\Vert_{W^{1, \infty}}\Vert \psi\Vert_{D(\L^{\alpha-\delta})}\Vert \psi\Vert_{D(\L^\alpha)}+C\Vert \nabla\phi\Vert_{W^{1, \infty}}\Vert \psi\Vert_{D(\L^\alpha)}\Vert \psi\Vert_{D(\L^\delta)}.
\eq
\end{lemm}
\begin{proof}
1. By \eqref{gradop}, 
\[
\Vert \L^{-1+\alpha}\nabla^\perp\psi\Vert_{L^2} \le\Vert \psi\Vert_{D(\L^\alpha)}.
\]
On the other hand,  a direct calculation gives
\[
\L^{-\alpha}[\Lambda^\alpha, \nabla\phi] \psi=[\nabla\phi, \L^{-\alpha}]\L^\alpha \psi
\]
which, by virtue of Theorem \ref{commu:new}, belongs to $D(\L)$ and satisfies
\[
\Vert \L[\nabla\phi, \L^{-\alpha}]\L^\alpha \psi\Vert_{L^2}\le C\Vert \nabla\phi\Vert_{W^{1, \infty}}\Vert \L^\alpha \psi\Vert_{L^2}= C\Vert \nabla\phi\Vert_{W^{1, \infty}}\Vert \psi\Vert_{D(\L^\alpha)}.
\]
Therefore, the integral defining $\cN_2(\psi, \phi)$ in \eqref{key2} makes sense and obeys the bound \eqref{cN2:bound}.

2. Let $\delta\in [0, \min(\alpha, 1-\alpha))$. According to \eqref{key2},
\[
\begin{aligned}
\cN_2(\psi, \phi)&=\langle \L^{-1+\alpha}\nabla^\perp\psi, \L^{1-\alpha}[\Lambda^\alpha, \nabla\phi] \psi \rangle_{L^2, L^2}\\
&=\langle \L^{-1+\alpha-\delta}\nabla^\perp\psi, \L^{1-\alpha+\delta}[\Lambda^\alpha, \nabla\phi] \psi \rangle_{D(\L^{\delta}), D(\L^{-\delta})}.
\end{aligned}
\]
Now we write
\begin{align*}
\L^{1-\alpha+\delta}[\Lambda^\alpha, \nabla\phi] \psi&=\L\L^{-\alpha+\delta}[\Lambda^\alpha, \nabla\phi] \psi\\
&=\L\left\{\L^\delta(\nabla \phi\psi)-\L^{-\alpha+\delta}(a\L^\alpha \psi)\right\}\\
&=\L\left\{[\L^\delta,\nabla \phi]\psi+\nabla\phi\L^\delta \psi-\L^{-\alpha+\delta}(a\L^\alpha \psi)\right\}\\
&=\L\left\{[\L^\delta,\nabla \phi]\psi+\nabla\phi\L^{-\alpha+\delta}\L^\alpha \psi-\L^{-\alpha+\delta}(a\L^\alpha \psi)\right\}\\
&=\L[\L^\delta, \nabla\phi]\psi+\L[\nabla\phi, \L^{-\alpha+\delta}]\L^\alpha \psi,
\end{align*}
where, according to \eqref{gain:1-s}, 
\[
[\L^\delta, \nabla\phi]\psi\in D(\L^{1-\delta}),
\]
 so 
 \[
 \L[\L^\delta, \nabla\phi]\psi\in D(\L^{-\delta});
 \]
 on the other hand, according to Theorem \ref{commu:new},
 \[
 \L[\nabla\phi, \L^{-\alpha+\delta}]\L^\alpha \psi\in L^2(\Omega)\subset D(\L^{-\delta}).
 \]
  Thus, we can write
\begin{align*}
I&=\langle \L^{-1+\alpha-\delta}\nabla^\perp\psi, \L[\nabla\phi, \L^{-\alpha+\delta}]\L^\alpha \psi \rangle_{D(\L^{\delta}), D(\L^{-\delta})}+\langle \L^{-1+\alpha-\delta}\nabla^\perp\psi, \L[\L^\delta, \nabla\phi]\psi \rangle_{D(\L^{\delta}), D(\L^{-\delta})}\\
&=\int_\Omega  \L^{-1+\alpha-\delta}\nabla^\perp\psi\cdot \L[\nabla\phi, \L^{-\alpha+\delta}]\L^\alpha \psi dx+\int_\Omega\L^{-1+\alpha}\nabla^\perp\psi\cdot \L^{1-\delta}[\L^\delta, \nabla\phi]\psi dx\\
&=\int_\Omega  \L^{-1+\alpha-\delta}\nabla^\perp\psi\cdot \L[\nabla\phi, \L^{-\alpha+\delta}]\L^\alpha \psi dx+\int_\Omega  \L^{-1+\alpha}\nabla^\perp\psi\cdot \L[\nabla\phi, \L^{-\delta}]\L^\delta\psi dx.
\end{align*}
As in 1., an application of Theorems \ref{commu:CN}, \ref{commu:new}, and \eqref{gradop} (with  $(\gamma=\alpha-\delta, \mu=-1-\alpha-\delta)$ and $(\gamma=\alpha, \mu=-1+\alpha)$) leads to the bound \eqref{cN2d:bound}.
\end{proof}
Let us denote
\bq
\begin{aligned}
\cN_1(\psi, \phi)&=\int_\Omega  [\Lambda^\alpha, \nabla^\perp]\psi\cdot \nabla\phi\psi dx,\\
\cN(\psi, \phi)&=\mez\cN_1(\psi, \phi)-\mez\cN_2(\psi, \phi).
\end{aligned}
\eq
Putting together the above considerations, we have proved that
\begin{lemm}\label{lemm:rep}
 If $\psi\in H^1_0(\Omega)$ then 
\[
\int_\Omega u\tt\cdot\nabla\phi=\cN(\psi, \phi).
\]
If $\tt\in L^2(\Omega)$ then
\[
\la \cN(\psi, \phi)\ra\le C\Vert \phi\Vert_{H^4}\Vert \psi\Vert_{L^2}^2+C\Vert \nabla\phi\Vert_{W^{1, \infty}}\Vert \psi\Vert_{D(\L^{\alpha})}^2
\]
and for any $\delta\in (0, \min(\alpha, 1-\alpha))$,
\[
\la \cN(\psi, \phi)\ra\le C\Vert \phi\Vert_{H^4}\Vert \psi\Vert_{L^2}^2+C\Vert \nabla\phi\Vert_{W^{1, \infty}}\Vert \psi\Vert_{D(\L^{\alpha-\delta})}\Vert \psi\Vert_{D(\L^\alpha)}+C\Vert \nabla\phi\Vert_{W^{1, \infty}}\Vert \psi\Vert_{D(\L^\alpha)}\Vert \psi\Vert_{D(\L^\delta)}.
\]
\end{lemm}
\subsection{Viscosity approximations}
 Let us fix $\tt_0\in L^2(\Omega)$ and a positive time $T$.  For each fixed $\eps>0$ we consider the viscosity approximation of \eqref{sSQG}:
\bq\label{sSQGe}
\begin{cases}
\partial_t\tt^\eps+u^\eps \cdot \nabla \tt^\eps-\eps \Delta \tt^\eps=0,&\quad t>0,\\
\tt^\eps=\tt_0, &\quad t=0
\end{cases}
\eq
with $u^\eps=\nabla^\perp\psi^\eps$, $\psi^\eps=\L^{-\alpha}\tt^\eps$.

Equation \eqref{sSQGe} can be solved using the Galerkin approximation method as follows. Denote by $\P_m$ the projection in $L^2(\Omega)$ onto the linear span $L^2_m(\Omega)$ of eigenfunctions $\{w_1,...,w_m\}$, {\it {\it i.e.}}
\[
\P_m f=\sum_{j=1}^mf_jw_j\quad\text{for}~f=\sum_{j=1}^\infty f_jw_j.
 \]
 We recall the following lemma which shows that for $\phi\in C_0^\infty(\Omega)$, $\P_m\phi$ are good approximations of $\phi$ in any Soblev space.
 \begin{lemm}[\protect{Lemma 3.1, \cite{ConNgu}}]\label{estPm}
 Let $\phi\in C_0^{\infty}(\Omega)$. For all $k\in \mathbb N$ we have
\bq\label{error:Pm}
\lim_{m\to\infty}\Vert\left(\mathbb I - \P_m\right)\phi\Vert_{H^k(\Omega)} = 0.
\eq
 \end{lemm}
The $m$th Galerkin approximation of \eqref{sSQGe} is the following ODE system in the finite dimensional space $\P_mL^2(\Omega) = L^2_m$:
\bq\label{Galerkin}
\begin{cases}
\dot \tt^\eps_m+\P_m(u^\eps_m\cdot\nabla\tt_m^\eps)-\eps\Delta\tt^\eps_m=0,&\quad t>0,\\
\tt_m^\eps=P_m\tt_0, &\quad t=0
\end{cases}
\eq
with $\tt_m(x, t)=\sum_{j=1}^m\tt_ j^{(m)} (t)w_j(x)$ and $u_m=\nabla^\perp\L^{-\alpha}\tt_m$ automatically satisfying $\cnx u_m=0$. Note that in general $u_m\notin L^2_m$.
The existence of solutions of \eqref{Galerkin} at fixed $m$ follows from the fact that this is an ODE:
\[
\frac{d\theta^{(m)}_l}{dt} + \sum_{j,k=1}^m\gamma^{(m)}_{jkl}\theta^{(m)}_j\theta^{(m)}_{k} +\eps\lambda_l\tt^{(m)}_l= 0
\label{galmode}
\]
with
\[
\gamma^{(m)}_{jkl} = \lambda_j^{\frac{-\alpha}{2}}\int_{\Omega}\left(\na^{\perp}w_j\cdot\na w_k\right)w_ldx.
\]
Since $\P_m$ is self-adjoint in $L^2$ , $u_m$ is divergence-free and $w_j$ vanishes at the boundary $\partial\Omega$, integration by parts with $\tt_m$ gives
\[
\int_\Omega \tt_m\P_m(u_m\cdot\nabla \tt_m)dx=\int_\Omega \tt_m u_m\cdot\nabla \tt_mdx=0
\]
and 
\[
-\int_\Omega \Delta\tt^\eps_m\tt^\eps_mdx=\int_\Omega \la\nabla\tt^\eps_m\ra^2dx.
\]
 It follows that 
 \[
 \mez\frac{d}{dt}\Vert \tt_m(\cdot, t)\Vert^2_{L^2(\Omega)}+\eps\Vert\nabla\tt^\eps_m\Vert_{L^2(\Omega)}^2=0,
 \]
  and thus for $t\in [0, T]$,
\bq\label{L^2bound}
\mez\Vert \tt^\eps_m(\cdot, t)\Vert^2_{L^2(\Omega)}+\eps\int_0^t\Vert\nabla\tt^\eps_m(\cdot, s)\Vert^2_{L^2(\Omega)}ds=\mez\Vert \tt_m^\eps(\cdot, 0)\Vert^2_{L^2(\Omega)}\le \mez \Vert \tt_0\Vert^2_{L^2(\Omega)}.
\eq
This can be seen directly on the ODE because $\gamma^{(m)}_{jkl}$ is antisymmetric in $k,l$. Therefore, the smooth solution $\tt_m^\eps$ of \eqref{Galerkin} exists globally and obeys the $L^2$ bound \eqref{L^2bound}. The sequence $(\tt^\eps_m)_m$ is thus uniformly in $m$ bounded in $L^\infty([0, T]; L^2(\Omega))\cap L^2([0, T]; H^1_0(\Omega))$. Consequently, for any $p\in [1, \infty)$ and any $q\in [1, \frac{2}{1-\alpha}]$, we have
\[
\begin{aligned}
&\tt_m^\eps\in L^2([0, T]; H^1_0(\Omega))\subset L^2([0, T]; L^p(\Omega)),\\
& u^\eps_m=\nabla^\perp\L^{-\alpha}\tt_m\in L^2([0, T]; H^\alpha(\Omega))\subset L^2([0, T]; L^q(\Omega))
\end{aligned}
\]
with bounds uniform with respect to $m$, where we have used Proposition \ref{prop:inject} to have 
\[
\L^{-\alpha}\tt_m\in L^2([0, T]; D(\L^{1+\alpha}))\subset L^2([0, T]; H^{1+\alpha}(\Omega)).
\]
 In particular, 
\bq
\label{reg:N:e}
\begin{aligned}
\Vert u^\eps_m\cdot \nabla\tt^\eps_m\Vert_{L^1([0, T]; H^{-1}(\Omega))}&=\Vert \cnx(u^\eps_m\cdot \tt^\eps_m)\Vert_{L^1([0, T]; H^{-1}(\Omega))}\\
&\le C\Vert \tt_m^\eps\Vert_{L^2([0, T]; H^1(\Omega))}^2\\
&\le \frac{C}{\eps}\Vert \tt_0\Vert_{H^1(\Omega)}^2
\end{aligned}
\eq
where \eqref{L^2bound} was invoked in the last inequality. Therefore, using \eqref{Galerkin} we obtain that $(\partial_t\tt^\eps_m)_m$ is uniformly in $m$ bounded in $L^1([0, T]; H^{-1}(\Omega))$. Then according to the Aubin-Lions lemma (\cite{Lions}), there exist a $\tt^\eps$,
\bq\label{reg:vis}
\tt^\eps\in L^\infty([0, T]; L^2(\Omega))\cap L^2([0, T]; H^1_0(\Omega)),
\eq
and a subsequence of $(\tt_m^\eps)_m$ such that
\bq\label{converge:vis}
\begin{aligned}
&\tt_m^\eps\to \tt^\eps\quad \text{strongly~in}~L^p([0, T]; H^{-\mu}(\Omega))\cap L^2([0, T]; H^{1-\mu}_0(\Omega))
\end{aligned}
\eq
for any $p<\infty$ and $\mu\in (0, 1)$.

Integrating by parts the first equation of \eqref{Galerkin} against any test function $\phi\in C^\infty_0(\Omega\times (0, T))$ gives
\bq\label{weak:approx} 
\int_0^T\int_\Omega\tt_m^\eps\partial_t\phi dxdt+\int_0^T \int_\Omega \tt_m^\eps u_m^\eps\cdot\nabla\P_m\phi(x, t)dxdt+\eps\int_0^T\int_\Omega \tt_m^\eps \Delta \phi dxdt=0.
\eq
In the limit $m\to \infty$, the first term and the third term converge repsectively to
\[
\int_0^T\int_\Omega\tt^\eps\partial_t\phi dxdt,\quad \eps\int_0^T\int_\Omega \tt^\eps\Delta \phi dxdt.
\]
It remains to study the nonlinear term:
\begin{align*}
N&:=\int_0^T \int_\Omega \tt_m^\eps u_m^\eps\cdot\nabla\P_m\phi dxdt-\int_0^T \int_\Omega \tt^\eps u_m^\eps\cdot\nabla\phi dxdt\\
&=\int_0^T \int_\Omega \tt_m^\eps u_m^\eps \cdot\nabla(\P_m\phi-\phi)dxdt+\int_0^T \int_\Omega (\tt_m^\eps-\tt^\eps) u_m^\eps\cdot\nabla\phi dxdt+\int_0^T \int_\Omega \tt^\eps(u_m^\eps-u^\eps)\cdot\nabla\phi dxdt\\
&=:N_1+N_2+N_3.
\end{align*}
 Lemma \ref{estPm} ensures that $\lim_{m\to \infty} N_1=0$. On the other hand, the strong convergence \eqref{converge:vis} with sufficiently small $\mu$ implies $
\lim_{m\to \infty}N_2=\lim_{m\to \infty}N_3=0$. Thus, we have proved that $\tt^\eps$ satisfies 
\[
\int_0^T\int_\Omega\tt^\eps \partial_t\phi dxdt+\int_0^T \int_\Omega \tt^\eps u^\eps \cdot\nabla\phi dxdt\\
+\eps\int_0^T\int_\Omega \tt^\eps \Delta \phi dxdt=0
\]
for all $\phi\in C^\infty_0(\Omega\times (0, T))$. Here, $\tt^\eps$ has the regularity \eqref{reg:vis}, and in view of \eqref{L^2bound},
\bq\label{uni:tt}
\Vert \tt^\eps\Vert_{L^\infty([0, T]; L^2(\Omega))}\le \Vert \tt_0\Vert_{L^2(\Omega)}.
\eq
Since $\psi^\eps\in D(\L^{1+\alpha})\subset H^1_0(\Omega)$, using Lemma \ref{lemm:rep} for the representation of the nonlinearity, we obtain for all $\phi\in C^\infty_0(\Omega\times (0, T))$,
\bq\label{weak:vis} 
\int_0^T\int_\Omega\tt^\eps \partial_t\phi dxdt+\int_0^T \cN(\psi^\eps, \phi)dt+\eps\int_0^T\int_\Omega \tt^\eps \Delta \phi dxdt=0.
\eq
Moreover, integrating by parts equation \eqref{Galerkin} with $\psi_m^\eps$ leads to
\[
\mez\frac{d}{dt}\Vert \psi^\eps_m(\cdot, t)\Vert_{D(\L^{\frac{\alpha}{2}})}^2+\eps\Vert \psi_m^\eps\Vert_{D(\L^{1+\frac{\alpha}{2}})}^2=0
\]
where we used the fact that the nonlinear term vanishes:
\[
\int_{\Omega}\psi_m^\eps\P_m(u_m^\eps\cdot \nabla \tt_m^\eps)dx=\int_{\Omega}\psi_m^\eps \cnx(\na^\perp\psi_m^\eps \tt_m)dx=-\int_{\Omega}\na \psi_m ^\eps\cdot\na^\perp\psi_m^\eps\tt_mdx=0.
\]
Consequently, integrating in time and letting $m\to \infty$ result in
\bq\label{cons:psie}
\Vert \psi^\eps(\cdot, t)\Vert_{D(\L^{\frac{\alpha}{2}})}^2+\eps\int_0^t\Vert \psi^\eps(\cdot, s)\Vert_{D(\L^{1+\frac{\alpha}{2}})}^2ds=\Vert \psi^\eps(\cdot, 0)\Vert_{D(\L^{\frac{\alpha}{2}})}^2\quad \forall t>0.
\eq
\subsection{Vanishing viscosity}  In order to extract a convergent subsequence of $\tt^\eps$ we need, in addition to \eqref{uni:tt}, a uniform bound for $\partial_t\tt^\eps$ in a lower norm. Let us note that the bound \eqref{reg:N:e} is not uniform in $\eps$. By \eqref{reg:vis}, $\tt^\eps \in D(\L)$ which implies $\psi^\eps=\L^{-\alpha}\tt^\eps\in D(\L^{1+\alpha})\subset D(\L)$. Lemma \ref{lemm:rep} then gives
\[
\la \int_\Omega \tt^\eps u^\eps\cdot \nabla \phi dx\ra \le C\Vert \phi\Vert_{H^4(\Omega)}\Vert \psi^\eps\Vert_{D(\L^\alpha)}^2\le C\Vert \phi\Vert_{H^4(\Omega)}\Vert \tt_0\Vert_{L^2(\Omega)}^2,
\]
and hence, in view of \eqref{weak:vis}, 
\[
\la \int_0^T\int_\Omega\tt^\eps \partial_t\phi dxdt\ra \le C\Vert \phi\Vert_{L^1([0, T]; H^4(\Omega))}(\Vert \tt_0\Vert_{L^2(\Omega)}+\Vert \tt_0\Vert_{L^2(\Omega)}^2)
\]
for all $\phi\in C^\infty_0(\Omega\times (0, T))$. Consequently,
\bq\label{uni:dt}
\Vert \partial_t \tt^\eps\Vert_{L^\infty([0, T]; H^{-4}(\Omega))}\le C(\Vert \tt_0\Vert_{L^2(\Omega)}+\Vert \tt_0\Vert_{L^2(\Omega)}^2).
\eq
In view of the uniform bounds \eqref{uni:tt} and \eqref{uni:dt}, the Aubin-Lions lemma ensures the existence of a $\tt$,
\[
\tt\in L^\infty([0, T]; L^2(\Omega))\cap C([0, T]; H^{-\nu}(\Omega))\quad \forall\nu>0,
\]
and a subsequence $\tt^\eps$ such that
\begin{align}
\label{weak:tte}
&\tt^\eps \rightharpoonup \tt \quad\text{weakly~ in}~ L^2([0, T]; L^2(\Omega)),\\
\label{strong:tte}
&\tt^\eps \to \tt\quad\text{strongly~in}~C([0, T]; H^{-\nu}(\Omega))\quad \forall \nu>0.
\end{align}
Consequently, with $\psi:=\L^{-\alpha}\tt$, 
\[
\psi \in L^\infty([0, T]; D(\L^{\alpha}))\cap C([0, T]; D(\L^{\alpha-\nu})\quad \forall \nu>0,
\]
we have 
\begin{align}
\label{weak:psie}
&\psi^\eps \rightharpoonup \psi \quad\text{weakly~ in}~ L^2([0, T]; D(\L^{\alpha})),\\
\label{strong:psie}
&\psi^\eps \to \psi\quad\text{strongly~in}~C([0, T]; D(\L^{\alpha-\nu}))\quad \forall \nu>0.
\end{align}
Let $\phi\in C^\infty_0(\Omega\times (0, T))$ a be fixed test function, we send $\eps$ to $0$ in the weak formulation \eqref{weak:vis}. The first term converges to $\int_0^T\int_\Omega \tt\partial_t\phi dx dt$ and the last term converges to $0$. Regarding the nonlinear term, we shall prove that
\[
R^\eps:=\int_0^T \cN(\psi^\eps, \phi) -\cN(\psi, \phi) dt
\]
converges to $0$. In view of \eqref{key}, \eqref{key3}, we have $2R^\eps=\sum_{j=1}^6 I^\eps_j$ with
\begin{align*}
&I^\eps_1=\int_\Omega  [\Lambda^{2-\beta}, \nabla^\perp](\psi_\eps-\psi)\cdot \nabla\phi\psi_\eps dx,\\
&I^\eps_2=\int_\Omega  [\Lambda^{2-\beta}, \nabla^\perp]\psi\cdot \nabla\phi(\psi_\eps-\psi)dx,\\
&I^\eps_3=-\int_0^T\int_\Omega  \L^{-1+\alpha-\delta}\nabla^\perp(\psi^\eps-\psi)\cdot \L[\nabla\phi, \L^{-\alpha+\delta}]\L^\alpha \psi^\eps dxdt\\
&I^\eps_4=-\int_0^T\int_\Omega  \L^{-1+\alpha-\delta}\nabla^\perp\psi\cdot \L[\nabla\phi, \L^{-\alpha+\delta}]\L^\alpha (\psi_\eps-\psi) dxdt,\\
&I^\eps_5=-\int_\Omega  \L^{-1+\alpha}\nabla^\perp(\psi_\eps-\psi)\cdot \L[\nabla\phi , \L^{-\delta}]\L^\delta\psi dx,\\
&I^\eps_6=-\int_\Omega  \L^{-1+\alpha}\nabla^\perp\psi_\eps\cdot \L[\nabla\phi , \L^{-\delta}]\L^\delta(\psi_\eps-\psi) dx
\end{align*}
with $\delta\in (0, \min(\alpha, 1-\alpha))$.

By virtue of Theorem \ref{commu:CN},
\[
|I^\eps_1|\le C(\phi)\Vert \psi_\eps-\psi\Vert_{L^2(\Omega)}\Vert \psi_\eps\Vert_{L^2(\Omega)},\quad |I^\eps_2|\le C(\phi)\Vert \psi_\eps-\psi\Vert_{L^2(\Omega)}\Vert \psi\Vert_{L^2(\Omega)}.
\]
Hence $\lim_{\eps\to 0} I^\eps_1=\lim_{\eps\to 0} I^\eps_2=0$ in view of the convergence \eqref{strong:psie} with $\nu<\alpha$.\\
As for \eqref{cN2d:bound},
\[
\la I^\eps_3\ra \le C\Vert \nabla\phi\Vert_{L^1([0, T]; W^{1, \infty})}\Vert \psi^\eps-\psi\Vert_{L^\infty([0, T]; D(\L^{\alpha-\delta}))}\Vert \psi^\eps\Vert_{L^\infty([0, T]; D(\L^\alpha))} 
\]
which combined with \eqref{strong:psie} leads to $\lim_{\eps\to 0}I^\eps_3=0$. Because $\L[\nabla\phi, \L^{-\alpha+\delta}]\L^\alpha$ is norm continuous from $L^2([0, T]; D(\L^\alpha))$ to $L^2([0, T]; (\Omega))$, it is weak-weak continuous, and thus $\lim_{\eps\to 0}I^\eps_4=0$ noticing that $\L^{-1+\alpha-\delta}\nabla^\perp\psi\in L^\infty([0, T]; L^2(\Omega))\subset  L^2([0, T]; L^2(\Omega))$. Similarly, $\lim_{\eps\to 0}I_5^\eps=0$ since $\L^{-1+\alpha}\nabla^\perp(\psi_\eps-\psi) \rightharpoonup 0$ in $L^2([0, T]; D(\L^{\alpha}))$ (by \eqref{weak:psie}) and $\L[\nabla\phi , \L^{-\delta}]\L^\delta\psi \in L^2([0, T]; L^2(\Omega))$ (by Theorem \ref{commu:new}). Finally, by Theorem \ref{commu:new}, 
\[
\la I_6^\eps\ra\le  \Vert \L^{-1+\alpha}\nabla^\perp\psi_\eps\Vert_{L^2(\Omega)} \Vert\L[\nabla\phi , \L^{-\delta}]\L^\delta(\psi_\eps-\psi)\Vert_{L^2(\Omega)}\le \Vert \psi_\eps\Vert_{D(\L^\alpha)} \Vert \psi_\eps-\psi\Vert_{D(\L^\delta)}\to 0
\]
noticing that $\delta<\alpha$. We conclude that
\[
\int_0^T\int_\Omega \tt\partial_t \phi dxdt+\int_0^T\cN(\psi, \phi)dt=0\quad\forall \phi\in C^\infty_0(\Omega\times (0, T)).
\]
Moreover, because of the strong convergence \eqref{strong:tte} the initial data is attained:
\[
\tt(\cdot, 0)=\lim_{\eps\to 0}\tt^\eps(\cdot, 0)=\lim_{\eps\to 0}\tt_0(\cdot)=\tt_0(\cdot)\quad\text{in}~H^{-\nu}(\Omega)\quad\forall \nu>0.
\]
Sending $\eps\to 0$ in \eqref{cons:psie} leads to the conservation 
\[
\Vert \psi(\cdot, t)\Vert_{D(\L^{\frac{\alpha}{2}})}=\Vert \psi(\cdot, 0)\Vert_{D(\L^{\frac{\alpha}{2}})}\quad \forall t>0.
\]
Finally, the energy inequality \eqref{energyin} follows from \eqref{uni:tt} and lower semicontinuity.
\begin{rema}\label{rema:method}
If we implement directly the Galerkin approximations for \eqref{sSQG} then  in view of \eqref{key}, we need to bound
\[
\la \int_\Omega  [\Lambda^{\alpha}, \nabla^\perp]\psi_m\cdot \nabla \P_m\phi\psi_mdx\ra.
\]
However, the commutator $[\Lambda^{\alpha}, \nabla^\perp]$ then cannot be bounded by means of Theorem \ref{commu:CN} because $\nabla \P_m\phi$ does not vanish on the boundary even though $\phi$ has compact support. In \cite{ConNgu}, we overcame this by first using Lemma \ref{estPm} and the fact that $u_m\tt_m$ is uniformly bounded in $L^1$ to approximate $
\int_\Omega u_m\tt_m\nabla\P_m\phi$ by $\int_\Omega u_m\tt_m\nabla\phi$. When $\alpha<1$, this argument breaks down since $u_m\tt_m$ is not anymore uniformly bounded in $L^1$. This explains why we proceeded the proof of Theorem \ref{main} using vanishing viscosity approximations.
\end{rema}
\section*{Appendix: Proof of Theorem \ref{commu:new}}\label{appendix}
In view of the identity 
\[
D^{-r}=c_r\int_0^\infty t^{-1+r}e^{-tD}dt
\] 
with $D, r>0$ we have the representation of negative powers of Laplacian via heat kernel:
\bq\label{frac}
\Lambda^{-s}f(x)=c_s\int_0^\infty t^{-1+\frac s2}e^{t\Delta} f(x)dt,\quad s>0.
\eq
Let $H(x, y, t)$ denote the heat kernel of $\Omega$, {\it {\it i.e.}}
\[
e^{t\Delta} f(x)=\int_\Omega H(x, y, t)f(y)dy\quad\forall x\in \Omega.
\]
We have from \cite{LiYau} the following bounds on $H$ and its gradient:
\bq\label{H}
H(x, y, t)\le Ct^{-\frac d2}e^{-\frac{|x-y|^2}{Kt}},
\eq
\bq\label{gradH}
|\nabla_x H(x, y, t)|\le Ct^{-\mez-\frac d2}e^{-\frac{|x-y|^2}{Kt}}
\eq
for all $(x,y)\in \Omega\times \Omega$ and $t>0$. 

We will also use the elementary estimate
\bq\label{integral}
\int_0^\infty t^{-1-\frac m2}e^{-\frac{p^2}{Kt}}dt\le C_{K, m}p^{-m},\quad m, p, K>0.
\eq
Let $f\in C^\infty_0(\Omega)$. Using \eqref{frac} we have
\bq\label{commu:new:f}
\begin{aligned}
\left[\L^{-s}, a\right] f(x)&=c_s\int_0^\infty t^{-1+\frac s2}\int_\Omega H(x, y, t)a(y)f(y)dt-c_sa(x)\int_0^\infty t^{-1+\frac s2}\int_\Omega H(x, y, t)f(y)dt\\
&=c_s\int_0^\infty t^{-1+\frac s2}\int_\Omega H(x, y, t)[a(y)-a(x)]f(y)dt.
\end{aligned}
\eq
In view of \eqref{H}, \eqref{integral}, and the assumption that $s<d$, we deduce that
\bq\label{cn:est}
\begin{aligned}
\la [\L^{-s}, a] f(x)\ra&\le C\Vert a\Vert_{L^\infty}\int_\Omega\int_0^\infty t^{-1+\frac{s}{2}-\frac d2}e^{-\frac{|x-y|^2}{Kt}} dt\la f(y)\ra dy\\
&\le C\Vert a\Vert_{L^\infty}\int_\Omega \frac{|f(y)|}{|x-y|^{d-s}}dy.
\end{aligned}
\eq
Let us recall the Hardy-Littlewood-Sobolev inequality. Let $\alpha\in(0, d)$ and $(p, r)\in (1, \infty)$ satisfy
\bq\label{HLS:exp}
\frac 1p+\frac \alpha d=1+\frac 1r.
\eq
A constant $C$ then exists such that
\bq\label{HLS}
\Vert f*|\cdot|^{-\alpha}\Vert_{L^r(\Rr^d)}\le C\Vert f\Vert_{L^p(\Rr^d)}.
\eq
Applying \eqref{HLS} with $\alpha=d-s$ leads to
\bq\label{comm:1}
\Vert [\L^{-s}, a] f\Vert_{L^r(\Omega)}\le C\Vert a\Vert_{L^\infty}\Vert f\Vert_{L^p(\Omega)}.
\eq
Let $\gamma_0$ denote the trace operator for $\Omega$. It is readily seen that $\gamma_0(\L^{-s}f)=0$ because $\L^{-s}f\in D(\L^m)$ for all $m\ge 0$, hence $\gamma_0(a\L^{-s}f)=\gamma_0(a)\gamma_0(\L^{-s}f)=0$. In addition, $a f\in H^1_0(\Omega)=D(\L)$, hence $\L^{-s}(af)\in D(\L^{1+s})\subset H^1_0(\Omega)$ and $\gamma_0(\L^{-s}(af))=0$. We deduce that 
\bq\label{trace:c}
\gamma_0([\L^{-s}, a] f)=0.
\eq
Next, for gradient bound we differentiate \eqref{commu:new:f}  and obtain
\begin{align*}
\nabla [\L^{-s}, a] f(x)&=c_s\int_0^\infty t^{-1+\frac s2}\int_\Omega \nabla_xH(x, y, t)[a(y)-a(x)]f(y)dt\\
&\quad -c_s\int_0^\infty t^{-1+\frac s2}\int_\Omega H(x, y, t)\nabla a(x)f(y)dt\\
&=:I +II.
\end{align*}
The term $II$ can be treated as above and we have
\bq\label{comm:2}
\Vert II\Vert_{L^r(\Omega)}\le C\Vert \nabla a\Vert_{L^\infty}\Vert f\Vert_{L^p(\Omega)}.
\eq
For $I$, we use the gradient estimate \eqref{gradH} for the heat kernel and the fact that 
\[
|a(x)-a(y)|\le \Vert \nabla a\Vert_{L^\infty }|x-y|
\]
 to arrive at
\[
\begin{aligned}
|I(x)|&\le C\Vert  \nabla a\Vert_{L^\infty } \int_\Omega \int_0^\infty t^{-1+\frac s2-\mez-\frac d2}e^{-\frac{|x-y|^2}{Kt}} dt |x-y||f(y)|dy\\
&\le C\Vert  \nabla a\Vert_{L^\infty} \int_\Omega \frac{|f(y)|}{|x-y|^{d-s}}dy.
\end{aligned}
\]
Appealing to \eqref{HLS} as before gives
\[
\Vert I\Vert_{L^r(\Omega)}\le C\Vert \nabla a\Vert_{L^\infty}\Vert f\Vert_{L^p(\Omega)}
\]
 which, combined with \eqref{comm:1}, \eqref{comm:2}, \eqref{trace:c}, leads to
\bq\label{cn1}
\Vert [\L^{-s}, a]f\Vert_{W^{1, r}_0(\Omega)}\le C\Vert a\Vert_{W^{1, \infty}(\Omega)}\Vert f\Vert_{L^p(\Omega)}
\eq
where $p, r$ satisfy \eqref{HLS} with $\alpha=d-s$.  Using the density of $C^\infty_0(\Omega)$ in $L^p(\Omega)$ for $p\in (1, \infty)$, and extension by continuity we conclude that the estimate \eqref{cn1} holds for any $f\in L^p(\Omega)$.

Now, for any $p\in (0, \infty)$, if $s<\frac{d}{p}$ then $r\in (1, \infty)$ given by 
\[
\frac{1}{r}=\frac{1}{p}-\frac{s}{d}
\]
satisfies \eqref{HLS}. Because $r>p$ and $\Omega$ is bounded, the continuous embedding $W^{1, r}_0(\Omega)\subset W^{1, p}_0(\Omega)$ yields
\bq\label{cn2}
\Vert [\L^{-s}, a]f\Vert_{W^{1, p}_0(\Omega)}\le C\Vert a\Vert_{W^{1, \infty}(\Omega)}\Vert f\Vert_{L^p(\Omega)}.
\eq
\section*{Appendix: Proof of Theorem \ref{commu:new2}}\label{appendix2}
In view of the identity 
\[
\lambda^{\frac s2}=c_s\int_0^\infty t^{-1-\frac s2}(1-e^{-t\lambda})dt
\] 
with $0< s<2$ and 
\[
1=c_s\int_0^\infty t^{-1-\frac s2}(1-e^{-t})dt
\]
we have the representation of the fractional Laplacian via heat kernel:
\bq\label{frac}
\Lambda^sf(x)=c_s\int_0^\infty t^{-1-\frac s2}(1-e^{t\Delta})f(x)dt,\quad 0<s<2.
\eq
Appealing to this representation, we have for $f\in C^\infty_0(\Omega)$
\[
[\L^s, a]f(x)=c_s\int_0^\infty t^{-1-\frac s2}\int_\Omega H(x, y, t)dt[a(x)-a(y)]f(y)dy.
\] 
In view of \eqref{H}, the fact that
\[
|a(x)-a(y)|\le \Vert a\Vert_{C^\gamma}|x-y|^\gamma,
\]
and \eqref{integral}, we deduce that
\begin{align*}
\la [\L^s, a]f(x)\ra &\le c_s \Vert a\Vert_{C^\gamma}\int_\Omega \int_0^\infty t^{-1-\frac s2-\frac d2}e^{-\frac{|x-y|^2}{Kt}} dt |x-y|^\gamma |f(y)|dy\\
&\le c_s \Vert a\Vert_{C^\gamma}\int_\Omega \frac{|f(y)|}{|x-y|^{d+s-\gamma}}dy.
\end{align*}
Then as in the proof of Theorem \ref{commu:new}, if $s<\gamma$ (note that $d+s-\gamma>0$), an application of the Hardy-Littlewood-Sobolev inequality leads to the bound \eqref{cmn2:e0}. Finally, \eqref{cmn2:e} follows from \eqref{cmn2:e0} and the fact that $\Omega$ is bounded.

{\bf{Acknowledgment.}} The author was partially supported by  NSF grant DMS-1209394. The author would like to thank Professor Peter Constantin for helpful comments on the paper.

\end{document}